\documentclass[11pt,reqno]{amsart}
\usepackage{amssymb}
\usepackage{amsmath}
\usepackage[active]{srcltx}
\usepackage{t1enc}
\usepackage[latin2]{inputenc}
\usepackage{verbatim}
\usepackage{amsmath,amsfonts,amssymb,amsthm}
\usepackage[mathcal]{eucal}
\usepackage{enumerate}
\usepackage[centertags]{amsmath}
\usepackage{graphics}
\usepackage{xcolor}

\newtheorem*{thm}{Theorem A } 
\newtheorem*{theo}{Theorem B}

\newtheorem{theorem}{Theorem}

\newtheorem{lemma}{Lemma}

\newtheorem{definition}{Definition}

\begin{document}
	\author{ G. Cagareishvili, V. Tsagareishvili, and G. Tutberidze}
	\title[The properties of general Fourier partial sums]{The properties of general Fourier partial sums of functions $f \in C_L$}

	\address{Giorgi Cagareishvili,Ivane Javakhishvili Tbilisi State University, Department of Mathematics, Faculty of Exact and Natural Sciences, Chavchavadze str. 1, Tbilisi 0128, Georgia }
	\email {giorgicagareishvili7@gmail.com}  
	
	\address{ Vakhtang Tsagareishvili, Ivane Javakhishvili Tbilisi State University, Department of Mathematics, Faculty of Exact and Natural Sciences, Chavchavadze str. 1, Tbilisi 0128, Georgia  }
	\email{cagare@ymail.com}
	
	\address{Giorgi Tutberidze, The University of Georgia, Institute of Mathematics, 77a Merab Kostava St, Tbilisi 0128, Georgia and Ivane Javakhishvili Tbilisi State University, Faculty of Exact and Natural Sciences, Chavchavadze str. 1, Tbilisi 0128, Georgia}
	\email{g.tutberidze@ug.edu.ge \and g.tutberidze@tsu.ge}

	\thanks{}
	\date{}
	\maketitle

	\begin{abstract}
		In this paper, we investigated the Fourier partial sums with respect to general orthonormal systems when the function $f$ belongs to some differentiable class of functions
	\end{abstract}
	
	
	
\section{Introduction}\label{sec1}

To maintain the flow of our discussion and the clarity of the proofs for our key findings, we have compiled all relevant notations, definitions, and preliminary concepts in Section 2.

As S. Banach proved that good differential properties of a function do not guarantee the a.e. boundedness of the Fourier partial sum of this function with respect to general orthonormal systems (ONS). In this paper we have studied the Fourier partial sums with respect to general ONS when function $f$ belongs to some differentiable class of functions. We also prove that this result is, in a sense, sharp. See Theorem \ref{t3}.

Concerning some recent such applications equipped with the present authors, we refer to the PhD. theses \cite{harpal} and \cite{tutberidze} and the references therein. In particular, some important engineering and industrial applications are investigated in these PhD. theses, e.g. concerning structural Health Monitoring, Artificial Intelligence, Neural Networks, Signal processing, Operational Model Analysis (OMA), Damage Detection of Bridges with focus on the impressive H\r alogaland Bridge in Narvik.

For our investigation, it is interesting that according to Menchov's and Banach's results, the convergence of general orthonormal series and the convergence of general Fourier series for functions from some differential classes are different problems. In the first case, the coefficients of the orthonormal series play a decisive role. In the second case, the fact that a function $f$ $(f \ne 0)$ belongs to any differential class does not guarantee the convergence of its Fourier series for general ONSs. Therefore, for the Fourier series, with respect to an ONS, to be convergent, one should impose additional conditions on the functions $\varphi_n$ of the ONS $(\varphi_n).$ Our main problem is to find the conditions for the functions  $\varphi_n$ of an ONS $(\varphi_n)$ for the Fourier series of every function of $Lip1$ classes to be bounded the Fourier partial sums.

Another motivation for this paper is to use our main results to obtain some new results concerning convergence/divergence of general Fourier series. Some other results for this case can be found in \cite{Edvards, GogoladzeTsagareishvili, GogoladzeTsagareishvili_5, GogoladzeTsagareishvili_6, KolmogorovFomin, PTT, Olevskii, Rademacher, Tandori, Cagareishvili,  TsaTutCa, tsatut3, tut1}. See also  the book \cite{ PTWbook} and monographs \cite{Alexits, KaczmarzSteinhaus, KashinSaakyan}.

The main results Theorems \ref{t2} and \ref{t3} are presented and proved in Section 3.

\section{Auxiliary definitions and results} \label{sec2}
By $Lip1$ we denote the class of functions $f$ from  $C(0,1)$, for which
$$\|f(x)-f(y)\|=O(h), \text{ when } \max |x-y| \leq h.$$
Let $C_L$ be the class of functions $f$ if 
$\frac{d}{dx}f \in Lip1.$

Suppose that $f\in L_2$ be an arbitrary function and $(\varphi_n)$  be ONS on $[0,1]$. Then the numbers
$$
C_n(f)=\int_0^1 f(x)  \varphi_n (x)\,dx, \quad n=1,2,\dots\,,
$$
are the Fourier coefficients of function $f$ with respect to the system $(\varphi_n)$ and
\begin{equation}\label{eq2}
	\sum_{n=1}^\infty   C_n (f)\varphi_n (x)
\end{equation}
is the Fourier series of this function $f$.

The general partial sum $S_n(x,f)$ of the series \eqref{eq2} is defined as follows
\begin{eqnarray}\label{eq2.2}
	S_n (x,f)=\sum_{k=1}^n  C_k (f)\varphi_k (x) .
\end{eqnarray}
Next, we can define $M_n(x)$ as follows $\left(x \in [0,1]\right)$ 

\begin{equation}\label{eq3}
	M_n(x)=\frac{1}{n} \sum_{i=1}^{n-1}\left|\int_{0}^{\frac{i}{n}}Q_n(u,x)du\right|,
\end{equation}
where
\begin{equation}\label{eq4}
	Q_n (u,x)=\sum_{k=1}^n  g_k (u)\varphi_k (x) 
\end{equation}
and
\begin{eqnarray}\label{eq1.2}
	g_k(u)=\int_{0}^{u} \varphi_k(t)dt. 
\end{eqnarray}

\begin{lemma} (see \cite{Cagareishvili})\label{l1}
	Let $(\varphi_n)$ be an arbitrary ONS on $[0,1]$. Then
	$$
	\frac{1}{n^2}  \sum_{k=1}^n  \varphi_k^2 (x)=O(1)n^{-\frac{1}{2}}  \;\;\text{a.e. on $[0,1]$}.
	$$
\end{lemma}

\textbf{Notation.} Let $G=[0,1] \setminus F,$ where 
\begin{eqnarray*}
	F = \left \{x\in[0,1] \left| \lim\limits_{n\to \infty}\frac{1}{n^2} \sum_{k=1}^{n} \varphi_k^2 (x) = \infty\right. \right \}.
\end{eqnarray*}
It is easy to show that
$$\left|F\right|=0 \text{  and  } \left|G\right|=1.$$

\begin{lemma}(see \cite{Cagareishvili})\label{l2}
	Suppose that
	\begin{equation}\label{eq5}
		B_n (u,x)=\sum_{k=1}^n  \varphi_k (u) \varphi_k (x).
	\end{equation}
	If for some $x\in [0,1]$
	\begin{equation}\label{eq6}
		\limsup_{n\to \infty} \bigg\vert \int_0^1  B_n (u,x)\,du\bigg\vert =+\infty,
	\end{equation}
	then for the function $q(u)=1$ $(u\in [0,1])$
	\begin{eqnarray*}
		\limsup_{n\to \infty} \vert S_n (x,q)\vert =+\infty ,
	\end{eqnarray*}
	where
	\begin{eqnarray*}
		S_n(x,q)=\sum_{k=1}^{n} C_k(q)\varphi_{k}(x) \text{ and } C_k(g)=\int_{0}^{1}q(u)\varphi_{k}(u)du = \int_{0}^{1} \varphi_{n}(u)du.
	\end{eqnarray*}

\end{lemma}
\begin{proof}
	Indeed, as
	$$C_n(q) = \int_{0}^{1}\varphi_{n}(u)du $$
	we get
	\begin{eqnarray*}
		\int_{0}^{1} B_n(u,x)du = \sum_{k=1}^{n}\int_{0}^{1}\varphi_{k}(u)du \varphi_{k}(x)=\sum_{k=1}^{n} C_k(q)  \varphi_{k}(x) = S_n(x,q).
	\end{eqnarray*}
	From the last equality and \eqref{eq6} we conclude that Lemma \ref{l2} is valid.
	
\end{proof}

\begin{lemma} (see \cite{Cagareishvili}) \label{l3}
	For any $i$ $(i=1,\dots,n)$ and $x\in [0,1]$ 
	$$
	\int_{\frac{i-1}{n}}^{\frac{i}{n}} \vert Q_n(u,x)\vert \,du\leq \frac{1}{n} \bigg(\sum_{k=1}^n \varphi_k^2(x)\bigg)^{\frac{1}{2}}
	$$
	holds (see \eqref{eq4}).
\end{lemma}
\begin{proof}
	By using \eqref{eq1.2} and Bessel inequality,
	\begin{eqnarray}
		\label{eq2.7}
		\sum_{k=1}^{\infty} g_k^2(u) = 	\sum_{k=1}^{\infty} \left(\int_{0}^{u} \varphi_k(t)dt\right)^2 \leq \int_{0}^u dt \leq 1	\end{eqnarray}
	where $u \in [0,1].$ Using \eqref{eq2.7}, the Cauchy and H\"older inequalities, we get $\left(i=1, 2, \dots , n\right)$ 
	\begin{eqnarray*}
		\int_{\frac{i-1}{n}}^{\frac{i}{n}} \vert Q_n(u,x)\vert \,du &\leq& \frac{1}{\sqrt{n}} \left(	\int_{\frac{i-1}{n}}^{\frac{i}{n}} Q_n^2(u,x) \,du\right)^{\frac{1}{2}} = \frac{1}{\sqrt{n}} \left(	\int_{\frac{i-1}{n}}^{\frac{i}{n}} \left(\sum_{k=1}^{n}g_k(u)\varphi_{k}(x)\right)^2 \,du\right)^{\frac{1}{2}} \\
		&\leq&	\frac{1}{\sqrt{n}} \left(	\int_{\frac{i-1}{n}}^{\frac{i}{n}} \sum\limits_{k=1}^{n}g_k^2(u)du \sum_{k=1}^n\varphi_{k}^2(x) \right)^{\frac{1}{2}} \leq \frac{1}{n}\left(\sum_{k=1}^n\varphi_{k}^2(x)\right)^{\frac{1}{2}}.
	\end{eqnarray*}
	The Lemma \ref{l3} is proved.
	
\end{proof}

\begin{lemma}(see \cite{Cagareishvili}) \label{l4}
	Let $(\varphi_n )$ be an ONS on $[0,1]$ and $f\in C_L$, then
	\begin{equation}\label{eq8}
		S_n (x,f)=f(1) \int_0^1  B_n (u,x)\,du-\int_0^1  f' (x) Q_n (u,x)\,du  .
	\end{equation}
\end{lemma}
\begin{proof}
	Integrating by parts, we obtain
	$$
	C_n (f)=\int_0^1  f(u) \varphi_k (u)\,du =f(1) \int_0^1  \varphi_k (u)\,du -\int_0^1  f' (u) g_n (u)\,du .
	$$
	Therefore,
	\begin{align}
		S_n (x,f) & =\sum_{k=1}^n  C_k (f)\varphi_k (x) \nonumber \\
		& =f(1) \int_0^1 \sum_{k=1}^n  \varphi_k (u) \varphi_k (x)\,du -\int_0^1  f'(u)\sum_{k=1}^n  g_k (u)\varphi_k (x)   \,du \nonumber \\
		& =f(1) \int_0^1  B_n (u,x)\,du-\int_0^1  f' (u) Q_n (u,x)\,du   . \label{eq9}
	\end{align}
	From \eqref{eq9} we derive \eqref{eq8}.
	
	The Lemma \ref{l4} is proved.
	
\end{proof}

\begin{lemma} \label{l5}
	Suppose that $q\in E(\varphi)$ $(q(u)=1, \ u\in [0,1])$ and
	\begin{eqnarray*}
		Q_n(u,x) = \sum_{k=1}^{n} g_k(u) \varphi_k(x).
	\end{eqnarray*}
	If for some $x \in [0,1]$
	\begin{eqnarray}
		\label{eq2.10}
		\limsup_{n \to \infty} \left|\int_{0}^{1} Q_n(u,x)du\right| = + \infty,
	\end{eqnarray}
	then for the function $p(u)=u$ $(u \in[0,1])$
	\begin{eqnarray*}
		\limsup_{n \to \infty} \left|S_n(x,p)\right| = + \infty.
	\end{eqnarray*}
\end{lemma}

\begin{proof}
	In equation \eqref{eq9} we suppose that $f = p = u$ $(u \in  [0, 1]),$ we receive
	\begin{eqnarray}
		S_n (x,h) = \int_0^1  B_n (u,x)\,du-\int_0^1 Q_n (u,x)\,du   . \label{eq2.11}
	\end{eqnarray}
	According to the condition of this lemma
	\begin{eqnarray*}
		\left|\int_0^1  B_n (u,x)\,du\right| = O(1).
	\end{eqnarray*}
	Consequently, using \eqref{eq2.10} and \eqref{eq2.11} we get
	\begin{eqnarray*}
		\limsup_{n \to \infty} \left|S_n(x,f)\right| = + \infty.
	\end{eqnarray*}
	Lemma \ref{l5} is proved.	
	
\end{proof}

\begin{definition}\label{d2}
	Let $E(\varphi )$ be a set of any functions $f$ with
	\begin{equation}\label{eq10}
		\vert S_n (x,f)\vert =O(1)
	\end{equation}
	at the point $x\in [0,1]$.
	\\
	
\end{definition}

\begin{thm}[S. Banach \cite{Banach}] \label{tA}
	For any $f\in L_2$ $(f\not\sim 0)$ there exists an ONS $(\varphi_n )$ such that
	\begin{eqnarray*}
		\limsup_{n\to \infty} \vert   S_n (x,f)\vert =+\infty \ \ \ \text{a.e. on } [0,1].
	\end{eqnarray*}	
	
\end{thm}

\begin{theo}[see \cite{2}] \label{tB}
	If   $f, F\in L_2$, then
	\begin{align}
		\int_0^1  f'(x)F(x)\,dx & =n\sum_{i=1}^{n-1}  \int_{\frac{i-1}{n}}^{\frac{i}{n}} \left(f(x)-f\left(x+\frac{1}{n} \right)\right)\,dx \int_0^{\frac{i}{n}}  F(x)\,dx \nonumber \\
		& \qquad  +   n\sum_{i=1}^{n-1}  \int_{\frac{i-1}{n}}^{\frac{i}{n}} \int_{\frac{i-1}{n}}^{\frac{i}{n}} \left(f(x)-f(u)\right)\,duF(x)\,dx \nonumber \\
		& \qquad  +n\int_{1-\frac{1}{n}}^1 f(x)\,dx\int_0^1  F(x)\,dx.\label{eq11}
	\end{align}
\end{theo}

\section{The main problems} \label{sec3}

From Theorem A it follows that even for function $g(x)=1$ $(x\in [0,1])$ there exists an ONS $(\varphi_n)$ such that
$$
\limsup_{n\to \infty} \vert S_n (x,g)\vert =+\infty
$$
a.e. on $[0,1]$.

\begin{theorem} \label{t2}
	Let $\left(\varphi_n\right)$ be an ONS on $[0, 1]$ and $p, q \in E\left(\varphi\right),$ where $p(u)=u, q(u)= 1,  u \in \left[0,1\right] .$
	If for $x \in G$
	$$M_n (x)=O(1),$$	
	then for any $f \in C_L,$
	$$S_n(x, f ) = O(1).$$
	
\end{theorem}

\begin{proof}
	If we substitute $F(x)=Q_n (u,x)$  and $f=f'$ in \eqref{eq11}, we obtain
	\begin{align}
		\int_0^1 f'(u)Q_n(u,x)\,du & =n \sum_{i=1}^{n-1} \int_{\frac{i-1}{n}}^{\frac{i}{n}} \left(f'(u)-f'\left(u+\frac{1}{n}\right)\right)\,du \int_0^{\frac{i}{n}}  Q_n(u,x)\,du \nonumber \\
		& \qquad +n\sum_{i=1}^{n-1}  \int_{\frac{i-1}{n}}^{\frac{i}{n}}  \int_{\frac{i-1}{n}}^{\frac{i}{n}} \left(f'(u)-f'(v)\right)\,dv \; Q_n (u,x)\,du  \nonumber \\ & \qquad +n\int_{1-\frac{1}{n}}^1 f'(u)\,du\; \int_0^1  Q_n (u,x)\,du \nonumber \\
		& =I_1+I_2+I_3. \label{eq13}
	\end{align}
	Due to \eqref{eq3} and the fact that $f'\in C_L$, we get $(\Delta _{in}=[\frac{i-1}{n},\frac{i}{n}])$
	\begin{align}\label{eq14}
		\vert  I_1 \vert &\leq  O\left(\frac{1}{n}\right) n \sum_{i=1}^{n-1} \int_{\Delta_{in}} \left|\int_{0}^{\frac{i}{n}}Q_n\left(u,x\right)du\right| dv \leq \frac{O(1)}{n} \sum_{i=1}^{n-1} \left|\int_{0}^{\frac{i}{n}}Q_n\left(u,x\right)du\right| \\
		& = O\left(1\right) M_n (x) = O(1). \notag
	\end{align}
	Next, since $f'\in C_L$ and $x\in G$, by using inequality  \eqref{eq2.7} from Lemma \ref{l3} and by using H\"older's and Parseval's $\left(\sum\limits_{k=1}^{n} g_k^2(u)\leq 1\right)$ inequalities, we have 
	\begin{align} \label{eq15} 
		\vert  I_2 \vert &\leq O\left(\frac{1}{n^2}\right) n \sum_{i=1}^{n} \int_{\Delta_{in}} \left|Q_n(u,x)\right|du =  O\left(1\right) \frac{1}{n}  \int_{0}^{1} \left| Q_n\left(u,x\right) \right| du \\
		&  = O(1) \frac{1}{n}  \left(\int_{0}^1 Q_n^2(u,x)du \right)^{\frac{1}{2}} = O(1) \frac{1}{n}  \left(\int_{0}^{1} \left(\sum_{k=1}^{n}g_k(u)  \varphi_k(x)\right)^2 du \right)^{\frac{1}{2}} \notag\\
		& =O(1) \frac{1}{n}  \left(\int_{0}^{1} \sum_{k=1}^{n}g_k^2(u)du \sum_{k=1}^{n} \varphi_k^2(x) \right)^{\frac{1}{2}}= O(1) \frac{1}{n}\left(\sum_{k=1}^{n} \varphi_k^2(x) \right)^{\frac{1}{2}}\notag\\
		&= O(1) \left(\frac{1}{n^2}\sum_{k=1}^{n} \varphi_k^2(x) \right)^{\frac{1}{2}} =O(1). \notag
	\end{align}
	Further, in \eqref{eq9} we suppose $f(u) = p(u) = u$ and $q(u)=1,$ $u \in [0,1].$ We receive 
	\begin{eqnarray*}
		S_n(x,p) = \int_{0}^{1}B_n(u,x)du-\int_{0}^{1}Q_n(u,x)du.
	\end{eqnarray*}
	From here as  
	\begin{eqnarray*}
		\int_{0}^{1}B_n(u,x)du=S_n(x,q).
	\end{eqnarray*}	
	we have  (see \eqref{eq5})  
	\begin{eqnarray*}
		S_n(x,p) = S_n(x,q) -\int_{0}^{1}Q_n(u,x)du.
	\end{eqnarray*}
	Consequently  if  $p, q \in E(\varphi),$ we have 
	\begin{eqnarray}\label{eq161}
		\vert  I_3 \vert &\leq& n\int_{1-\frac{1}{n}}^1  \vert  f'(u)\vert \,du \;\bigg\vert \int_0^1  Q_n (u,x)\,du\bigg\vert \\
		&\leq& n \ \frac{1}{n} \ \max_{i \in [0,1]} \vert  f'(u)\vert \bigg\vert \int_0^1  Q_n (u,x)\,du\bigg\vert = O(1). \notag
	\end{eqnarray}
	Taking the evaluations of $I_1,$ $I_2,$ and $I_3$ into account in \eqref{eq13} we conclude
	\begin{equation}\label{eq17}
		\bigg\vert  \int_0^1  f' (u) Q_n (u,x)\,du\bigg\vert  = O(1).
	\end{equation}
	Finally, according to the conditions of this theorem and \eqref{eq8}, we can deduce that
	$$ S_n (x,f) =O(1).$$
	Theorem \ref{t2} is completely proved.
	
\end{proof}

\begin{theorem}\label{t3}
	Let $(\varphi_n)$ be an ONS on $[0,1]$. If for some $t\in G$,
	$$
	\limsup_{n\to \infty} M_n (t)=+\infty,
	$$
	then there exists a function $r'\in C_L$ such that
	$$
	\limsup_{n\to \infty} \vert S_n (t,r)\vert =+\infty  .
	$$
\end{theorem}

\begin{proof}
	Firstly, according to Lemma \ref{l2} and Lemma \ref{l5}
	\begin{eqnarray}
		\label{eq3.4}
		\left| \int_0^1 B_n (u,t)\,du\right| =O(1) \text{ and } 
		\left| \int_0^1 G_n (u,t)\,du\right| =O(1)
	\end{eqnarray}
	otherwise, Theorem \ref{t3} is proved.
	
	We defined the sequence of functions $(f_n)$ as follows:
	\begin{eqnarray} \label{eq22}
		f_n(u)=\int_{0}^{u} sign \int_{0}^{y} Q_n(v,t)dv dy, \ \ \ n=1,2, \dots .
	\end{eqnarray}
	In \eqref{eq161} we substitute $Q_n\left(u,x\right)=Q_n\left(u,t\right)$ and $f'(u)= f_n (u),$ then
	\begin{eqnarray}
		\label{eq23} \int_{0}^{1} f_n (u)Q_n (u, t) du
		&=& \sum_{i=1}^{n-1}\left(f_n\left(\frac{i}{n}\right)-f_n\left(\frac{i+1}{n}\right)\right)\int_{0}^{\frac{i}{n}}Q_n(u,t)du \notag   \\
		&+& \sum_{i=1}^{n}\int_{\frac{i-1}{n}}^{\frac{i}{n}}\left(f_n\left(u\right) -f_n\left(\frac{i}{n}\right)\right)Q_n(u,t) du    \\
		&+& f_n\left(1\right)\int_{0}^{1}Q_n(u,t)du=S_1 +S_2  +S_3.    \notag
	\end{eqnarray}
	By using \eqref{eq22}, Lemma \ref{l3} and  H\"older's and Cauchy's inequality, we receive (see $|I_2|$)
	\begin{eqnarray} \label{eq27}
		\left|S_2\right| &\leq& \frac{1}{n} \sum_{i=1}^{n} \int_{\frac{i-1}{n}}^{\frac{i}{n}}\left|Q_n(u,t)\right|du = \frac{1}{n}\int_{0}^{1}\left|Q_n(u,t)\right|du \\
		&\leq&  O(1) \frac{1}{n} \left(\int_{0}^{1}  \left(\sum_{k=1}^{n}g_k(u)\varphi_{k}(t)\right)^2 du\right)^{\frac{1}{2}} = O(1). \notag
	\end{eqnarray}
	
	Afterwards taking into account \eqref{eq3.4} and \eqref{eq22} we get
	\begin{eqnarray}
		\label{eq28}
		\left|S_3\right| = O(1)\left|\int_{0}^{1}Q_n(u,t)du\right| = O(1).
	\end{eqnarray}

	Let $D_n$ be a set of all $i$ $\left(i=1,2,...,n-1\right)$, for all of which, there exists a point $t\in\left[\frac{i-1}{n},\frac{i}{n}\right]$ such, that 
	\begin{equation}
		sign\int_{0}^{\frac{i}{n}}Q_n \left(u,t\right) du \ne sign\int_{0}^{t}Q_n \left(u,t\right) du. \label{2}
	\end{equation}
	
	Suppose that $i\in D_n.$ On account  continuity of  function  $\int_{0}^{t}Q_n \left(u,x\right) du$ on $[0, 1]$ for some $t_{i_n}\in\left[\frac{i}{n},\frac{i+1}{n}\right]$ we have
	\begin{equation*}
		\int_{0}^{t_{i_n}}Q_n \left(u,t\right) du=0.
	\end{equation*}
	Consequently,
	\begin{eqnarray}
		\left|\int_{0}^{\frac{i}{n}}Q_n \left(u,t\right) dx\right|= \left|\int_{0}^{t_{i_n}}Q_n \left(u,t\right) du +\int_{t_{i_n}}^{\frac{i}{n}}Q_n \left(u,t\right) du\right| \leq\int_{t_{i_n}}^{\frac{i}{n}} \left|Q_n \left(u,t\right)du\right|.   \notag
	\end{eqnarray}
	Further (see \eqref{eq2.7}), by using H\"older's inequality, we have
	\begin{eqnarray*}
		&& \sum_{i\in D_n}\left|\int_{0}^{\frac{i}{n}}Q_n \left(u, t\right)du \right| \leq	\sum_{i=1}^{i-1}\left|\int_{t_{i_n}}^{\frac{i}{n}}Q_n \left(u, t\right)du \right| \leq \int_{0}^{1}\left|Q_n \left(u,t\right)\right|du \\
		&& \qquad \leq \left(\int_{0}^{1}Q_{n}^{2} \left(u, t\right) du\right)^{\frac{1}{2}}=\left(\int_{0}^{1}\left(\sum_{k=1}^{n}g_k(u)\varphi_{k}(t)\right)^2 du\right)^{\frac{1}{2}}  \\
		&& \qquad = \left(\int_{0}^{1}\sum_{k=1}^{n}g_k^2(u) du \sum_{k=1}^{n}\varphi_{k}^2(t) \right)^{\frac{1}{2}} \leq \left(\sum_{k=1}^{n}\varphi_{k}^2(t) \right)^{\frac{1}{2}}.
	\end{eqnarray*}
	Then (see Lemma \ref{l1})
	\begin{eqnarray}
		\frac{1}{n} \sum_{i\in D_n}\left|\int_{0}^{\frac{i}{n}}Q_n \left(u, t\right)du \right| \leq O(1)\frac{1}{n}\left( \sum_{k=1}^{n} \varphi_k^2(t)\right)^{\frac{1}{2}} = O(1).\label{eq3.10}
	\end{eqnarray}
	
	At present we denote $F_n =\{1, 2, \dots n-1 \} \setminus D_n.$ Suppose that $i\in F_n.$ then according to definition of $D_n$  we have 
	\begin{eqnarray*}
		\left(f_n\left(\frac{i}{n}\right)-f_n\left(\frac{i+1}{n}\right)\right)\int_{0}^{\frac{i}{n}}Q_n(u,t)du =-\frac{1}{n} \left|\int_{0}^{\frac{i}{n}}Q_n(u,t) du\right|.
	\end{eqnarray*}
	According to last equality, we have
	\begin{eqnarray}
		\label{eq24} 
		\left|\sum_{i\in F_n} \left(f_n \left(\frac{i}{n}\right) -f_n \left(\frac{i+1}{n}\right)\right)\int_{0}^{\frac{i}{n}}  Q_n(u,t)du\right| = \frac{1}{n} \sum_{i\in F_n} \left| \int_{0}^{\frac{i}{n}} Q_n(u,t)du\right|.   
	\end{eqnarray}
	
	According to  \eqref{eq3.10} and \eqref{eq27} we obtain 
	\begin{eqnarray} \label{eq151}
		\left|S_1\right| &&= \left| \sum_{i=1}^{n-1}\left(f_n\left(\frac{i}{n}\right) -f_n\left(\frac{i+1}{n}\right)\right)\int_{0}^{\frac{i}{n}}Q_n(u,t)du \right|   \\ 
		&& \geq \left|\sum_{i\in F_n} \left(f_n\left(\frac{i}{n}\right) - f_n\left(\frac{i+1}{n}\right)\right) \int_{0}^{\frac{i}{n}}Q_n(u,t)du \right| \notag \\
		&& -  \left|\sum_{i\in D_n} \left(f_n\left(\frac{i}{n}\right) - f_n\left(\frac{i+1}{n}\right)\right) \int_{0}^{\frac{i}{n}}Q_n(u,t)du \right| \notag \\
		&& \geq  \frac{1}{n} \sum_{i\in F_n} \left|\int_{0}^{\frac{i}{n}}Q_n(u,t)du\right|- \frac{1}{n} \sum_{i\in D_n} \left|\int_{0}^{\frac{i}{n}}Q_n(u,t)du\right| \notag \\
		&& \frac{1}{n} \sum_{i=1}^{n-1} \left|\int_{0}^{\frac{i}{n}}Q_n(u,t)dt\right|- \frac{2}{n} \sum_{i\in D_n} \left|\int_{0}^{\frac{i}{n}}Q_n(u,t)du\right|  \notag \\
		&&\geq M_n(t)- O (1).\notag
	\end{eqnarray}
	Finally, from \eqref{eq23} using \eqref{eq27}, \eqref{eq28} and \eqref{eq151},  we have
	\begin{eqnarray*}
		\left|\int_{0}^{1}f_n(t)Q_n(u,t)du\right| \geq M_n(t)-O(1).
	\end{eqnarray*}
	The condition of Theorem  \ref{t3} imply 
	\begin{eqnarray}	\label{eq29} 
		\limsup_{n \to \infty} \left|\int_{0}^{1}f_n(t)Q_n(u,t)du\right|=+\infty.
	\end{eqnarray}
	Consider the sequence of liner and bounded functionals on the Banach space $Lip1$
	\begin{eqnarray*}
		U_n(f) = \int_{0}^{1}f_n(u)Q_n(u,t)du.
	\end{eqnarray*}
	By \eqref{eq29}
	\begin{eqnarray*}
		\limsup_{n \to \infty}U_n(f_n) = + \infty.
	\end{eqnarray*}
	On the other hand                
	\begin{eqnarray*}
		\left|\left|f_n\right|\right|_{Lip1}=	\left|\left|f_n\right|\right|_{C} + \sup_{x,y\in \left[0,1\right]} \frac{\left|f_n(x)-f_n(y)\right|}{\left|x-y\right|} \leq 2. 
	\end{eqnarray*}
	Consequently (see \eqref{eq29}), according to the Banach-Steinhaus theorem, there exist such a function $h\in Lip1$ that
	\begin{eqnarray}	\label{eq30} 
		\limsup_{n \to \infty} \left|\int_{0}^{1}h(u)Q_n(u,t)du\right|=+\infty.
	\end{eqnarray}
	Let
	\begin{eqnarray*}
		m(u)=\int_{0}^{u}h(v)dv,
	\end{eqnarray*}
	using lemma \ref{l4} we get
	\begin{eqnarray*}
		S_n (t,m)= m(1) \int_0^1  B_n (u,t)\,du-\int_0^1  h(u) Q_n (u,t)\,du.
	\end{eqnarray*}
	From \eqref{eq3.4} and \eqref{eq30}  we get
	\begin{eqnarray*}
		\limsup_{n \to \infty}\left|S_n(t,m)\right|=+\infty.
	\end{eqnarray*}
	As $m' = h \in Lip1$ Theorem \ref{t3} is proved.
	
\end{proof}

Now we show that the condition of Theorem \ref{t2}  $( q,p \in E(\varphi), \ q(u)=1, \ p(u)=u, \ u\in [0,1])$ don't guarantee that
\begin{eqnarray*}
	\limsup_{n \to \infty}\left|S_n(t,f)\right|<\infty,
\end{eqnarray*}
for any function $f \in C_L.$

Indeed
\begin{theorem}
	\label{theorem4} 
	There exists a function $g \in C_L$ and ONS $(G_n)$  such that 
	\begin{eqnarray*}
		\int_{0}^{1} G_n (u) du =0, \ \ \ \   \int_{0}^{1} u \ G_n (u) du =0,  \ \ \ n=1, 2, \dots, 
	\end{eqnarray*}
	and
	\begin{eqnarray*}
		\limsup_{n \to \infty}\left|S_n(x,g,G)\right|= \limsup_{n \to \infty}\left|	\sum_{k=1}^{n} C_n(g,G) G_n(x)\right|=+ \infty, 
	\end{eqnarray*}
	where 
	\begin{eqnarray*}
		C_n(g,G)= \int_{0}^{1} g(u) G_n (u) du =0,  \ \ \ (n=1, 2, \dots)
	\end{eqnarray*}
	and
	\begin{eqnarray*}
		S_n(x,g,G) = 	\sum_{k=1}^{n} C_n(g,G) G_n(x).
	\end{eqnarray*}

\end{theorem}

\begin{proof}
	Let us assume that $f(x)=1-\cos 4(u- 1/2) \pi.$ According to the Banach Theorem there exist an ONS $(\varphi_n),$ such that a.e. on $[0,1]$
	\begin{eqnarray}
		\label{eq31} \limsup_{n \rightarrow +\infty} \left|S_n (x, f, \varphi)\right| =+\infty.
	\end{eqnarray}
	The system   $\varPhi_n(u)$  we define as follows
	
	\begin{equation*}
		\varPhi_n\left( u\right) =\left\{ 
		\begin{array}{ccc}
			\varphi_n(2u), & \text{when} & u\in \left[ 0,\frac{1}{2}\right), \\ 
			\\
			-\varphi_n\left(2\left(u-\frac{1}{2}\right)\right), & \text{when} & u\in \left[ \frac{1}{2},1\right].
		\end{array}%
		\right.
	\end{equation*}
	
	It is easy to prove that $\left(\varPhi_n\right)$ is an ONS on $[0,1]$ and $\int_{0}^1 \varPhi_n (u) du =0 , \ n=0,1, \dots .$
	
	Now we investigate the function
	\begin{equation*} \label{eq33}
		g\left( u\right) =\left\{ 
		\begin{array}{ccc}
			f(2u), & \text{when} & u\in \left[ 0,\frac{1}{2}\right), \\ 
			\\
			0, & \text{when} & u\in \left[ \frac{1}{2},1\right].
		\end{array}%
		\right.
	\end{equation*}
	We have
	\begin{eqnarray*}
		C_n(g,\varPhi)&=& \int_{0}^{1} g(u) \varPhi_n (u) du =\int_{0}^{\frac{1}{2}} f\left(2u\right) \varphi_n\left(2u\right)du\\
		&=& \frac{1}{2}\int_{0}^{1} f\left(u\right) \varphi_n\left(u\right)du = \frac{1}{2} C_n\left(f, \varphi\right).
	\end{eqnarray*}
	Consequently (see \eqref{eq31})
	\begin{eqnarray}
		\label{eq3.15} \limsup_{n \rightarrow +\infty} \left|S_n (x, f, \varPhi)\right| = \frac{1}{2} \limsup_{n \rightarrow +\infty} \left|S_n (x, f, \varphi)\right|=+\infty.
	\end{eqnarray}
	
	Now we define the next ONS $G_n(u)$ as follows
	
	\begin{equation} \label{eq34}
		G_n\left( u\right) =\left\{ 
		\begin{array}{ccc}
			\varPhi_{n}(2u), & \text{when} & u\in \left[ 0,\frac{1}{2}\right), \\ 
			\\
			- \varPhi_{n} \left(2\left(u-\frac{1}{2}\right)\right), & \text{when} & u\in \left[ \frac{1}{2},1\right].
		\end{array}%
		\right.
	\end{equation}
	After we define 	
	\begin{equation*} \label{eq33}
		h\left( u\right) =\left\{ 
		\begin{array}{ccc}
			g(2u), & \text{when} & u\in \left[ 0,\frac{1}{2}\right), \\ 
			\\
			0, & \text{when} & u\in \left[ \frac{1}{2},1\right].
		\end{array}%
		\right.
	\end{equation*}
	We have
	\begin{eqnarray*}
		C_n(h,G)&=& \int_{0}^{1} h(u) G_n (u) du =\int_{0}^{\frac{1}{2}} g\left(2u\right) \varPhi_n\left(2u\right)du\\
		&=& \frac{1}{2}\int_{0}^{1} g\left(u\right) \varPhi_n\left(u\right)du = \frac{1}{2} C_n\left(g, \varPhi\right).
	\end{eqnarray*}
	Such we have that $h'\in Lip1$ and  (see \eqref{eq31})
	\begin{eqnarray}
		\label{eq3.16} \limsup_{n \rightarrow +\infty} \left|S_n (x, h, G)\right| = \frac{1}{2} \limsup_{n \rightarrow +\infty} \left|S_n (x, g, \varPhi)\right|=+\infty.
	\end{eqnarray}
	Let $p(u)= u  ,$ we get
	\begin{eqnarray*}
		C_n(p,G)&=& \int_{0}^{1} u\ G_n (u) du =\int_{0}^{\frac{1}{2}}u\ \varPhi_n\left(2u\right)du - \frac{1}{2} \int_{0}^{1} u\ \varPhi_{n}\left(2\left(u-\frac{1}{2}\right)\right) du  \\
		&=&\frac{1}{4}\int_{0}^{1} u\ \varPhi_n\left(u\right)du - \frac{1}{2}\int_{0}^{1} \left({u + \frac{1}{2}}\right) \varPhi_n\left(u\right)du = -\frac{1}{4} C_n\left(q, \varPhi\right)=0.
	\end{eqnarray*}
	Finally we receive \\
	1) If $q(u)=1$ and $p(u)=u$ when $u \in [0,1],$ then
	\begin{eqnarray*}
		C_n\left(p,G\right) = C_n\left(q, G\right) = 0 \text{ or } 	\limsup_{n \rightarrow +\infty} \left|S_n (x, p, G)\right| = \frac{1}{2} \limsup_{n \rightarrow +\infty} \left|S_n (x, q, G)\right|<+\infty.
	\end{eqnarray*}	
	And (see \eqref{eq3.16})\\
	2) $\limsup\limits_{n \rightarrow +\infty} \left|S_n (x, h, G)\right| = +\infty \text{ where } h' \in Lip1. $
	
	Theorem \ref{theorem4} is completely proved.
	
\end{proof}

\section{Problems of efficiency} \label{sec4}
\begin{theorem}
	\label{t5}
	Let $\varphi_{n}(u) = \sqrt{2} \cos 2 \pi n u,$ then for any $x\in [0,1]$
	
	\begin{eqnarray*}
		M_n(x)= O (1).
	\end{eqnarray*}
\end{theorem}
\begin{proof}
	Indeed
	\begin{eqnarray*}
		M_n(x) &=& \frac{1}{n} \sum_{i=1}^{n-1} \left|\int_{0}^\frac{i}{n} Q_n(u,x)du\right|=\frac{2}{n} \sum_{i=1}^{n-1} \left|\int_{0}^\frac{i}{n} \sum_{k=1}^{n} \int_{0}^{u} \cos 2 \pi k v dv du \cos 2 \pi k x \right| \\
		&=& O(1) \max_{1 \leq i \leq n} \left|\int_{0}^\frac{i}{n} \frac{1}{2 \pi} \sum_{k=1}^{n}  \frac{1}{k} \sin 2 \pi k u du \cos 2 \pi k x \right| \\
		& = &O(1) \left(\int_{0}^{1} \sum_{k=1}^{n} \frac{1}{k^2}\right)^{\frac{1}{2}} = O(1).
	\end{eqnarray*}
	
\end{proof}

\begin{theorem}
	\label{t6}
	Let $(X_n)$ be Haar system (see \cite{Alexits},Ch.2) then for any $x \in [0,1]$
	\begin{eqnarray*}
		M_n(x) = O(1).
	\end{eqnarray*}
	
\end{theorem}
\begin{proof}
	If $2^s<m\leq2^{s+1}$ then 
	\begin{eqnarray*}
		\left|\int_{0}^{u}X_m(v) dv\right| \leq 2^{-\frac{s}{2}}
	\end{eqnarray*}
	and $(i=1,2, \dots ,n-1)$
	\begin{eqnarray*}
		\left| \int_{0}^{\frac{i}{n}} \int_{0}^{u} \sum_{m=2^s +1}^{2^{s+1}} X_m(v)dvX_m(x)du\right| \leq 2 \ 2^{-s}.
	\end{eqnarray*}	
	Finally 
	\begin{eqnarray*}
		M_n(x) &=& \frac{1}{n} \sum_{i=1}^{n-1} \left|\int_{0}^\frac{i}{n} Q_n(u,x)du\right|  \\
		&=& O(1)  \max_{1 \leq i \leq n} \left|\int_{0}^\frac{i}{n} \sum_{s=0}^{d} \int_{0}^{u} \sum_{m=2^s +1}^{2^{s+1}} X_m(v)dvX_m(x)du\right| = O(1).
	\end{eqnarray*}

	From here it is evident that Theorem \ref{t6} is valid.
	
\end{proof}

\section{Conclusion}\label{sec13}

Based on the discussion presented in this article, it is evident that although Fourier partial sums related to general orthonormal systems do not converge for all functions $f$ classified under a differentiable class of $ C_L$ functions, we can identify a specific subset of orthonormal systems. These subsets include functions that meet particular criteria, ensuring that the Fourier partial sums for $ C_L$ class functions exhibit convergence (refer to Theorem \ref{t2}). Moreover, we have established that the criteria applied to the functions of the orthonormal systems are both precise and reliable.

Additionally, it is important to note that each orthonormal system encompasses a subsystem for which the general Fourier series of any function $ f$ within the $ C_L$ class converges almost everywhere on the interval $[0,1].$ 

\section*{Declarations}

\begin{itemize}
	\item Funding: The research of third author is supported by Shota Rustaveli National Science Foundation grant no. FR-24-698;
	\item Conflict of interest/Competing interests: Author Giorgi Cagareishvili, Author Vakhtang Tsagareishvili, and Author Giorgi Tutberidze declare that they have no competing interests;
	\item Ethics approval and consent to participate: Not applicable;
	\item Consent for publication: Not applicable;
	\item Data availability: Contact the corresponding author for data requests;
	\item Materials availability: Contact the corresponding author for data requests;
	\item Code availability: Not applicable;
	\item Author contribution: All authors whose names appear on the submission contributed equally to this work. They made substantial contributions to this work, approved the version to be published, and agreed to be accountable for all aspects of the work and ensure that questions related to the accuracy or integrity of any part of the work are appropriately investigated and resolved.
\end{itemize}

\end{document}